\documentclass[10pt]{amsart}
\usepackage{graphicx,amssymb,amsfonts,amsmath,amsthm,newlfont}
\usepackage{mathrsfs,bbm}
\usepackage{epsfig,url}
\usepackage{color}

\usepackage[all,2cell]{xy} \UseAllTwocells \SilentMatrices

\newtheorem{thm}{Theorem}[section]
\newtheorem{cor}[thm]{Corollary}
\newtheorem{lem}[thm]{Lemma}
\newtheorem{prop}[thm]{Proposition}
\theoremstyle{definition}
\newtheorem{defn}[thm]{Definition}

\theoremstyle{remark}
\newtheorem{rem}[thm]{Remark}

\numberwithin{equation}{section}

\newcommand{\E}{\mathcal E}
\newcommand{\Or}{\mathcal O}
\newcommand{\M}{\mathcal M}

\newcommand{\VV}{\mathscr V}
\newcommand{\VVbar}{\overline{\VV}}

\newcommand{\HH}{\mathfrak{H}}

\newcommand{\An}{\mathsf A}
\newcommand{\Sn}{\mathsf S}
\newcommand{\Tn}{\mathsf T}

\newcommand{\A}{\mathbb A}
\newcommand{\C}{\mathbb C}
\newcommand{\R}{\mathbb R}
\newcommand{\DD}{\mathbb D}
\newcommand{\G}{\mathbb G}
\renewcommand{\P}{\mathbb P}
\newcommand{\Q}{\mathbb Q}
\newcommand{\V}{\mathbb V}
\newcommand{\Z}{\mathbb Z}

\newcommand{\mmu}{\boldsymbol{\mu}}

\newcommand{\bw}{\varpi}
\newcommand{\kk}{{\mathbbm k}}
\newcommand{\D}{\mathcal{D}} 

\newcommand{\PD}{\mathrm{P}} 

\newcommand{\ba}{\mathbf{a}}
\newcommand{\bb}{\mathbf{b}}
\newcommand{\e}{\mathbf{e}}

\newcommand{\adual}{\check{\ba}}
\newcommand{\bdual}{\check{\bb}}

\newcommand{\GE}{\mathsf{G}}		
\newcommand{\Mbar}{\overline{\M}}
\newcommand{\Ftilde}{\widetilde{F}}

\newcommand{\xitilde}{\tilde{\xi}}

\newcommand{\gr}{\mathrm{gr}}
\newcommand{\SL}{\mathrm{SL}}
\newcommand{\Gm}{{\G_m}}

\newcommand{\dR}{{dR}}
\newcommand{\cusp}{\mathrm{cusp}}
\newcommand{\cdR}{{\cusp,\dR}}
\newcommand{\an}{{\mathrm{an}}}

\newcommand{\vv}{{\vec{v}}}

\newcommand\bbs{{\backslash\negthickspace \backslash}}
\newcommand{\bdot}{\bullet}
\newcommand{\blank}{\phantom{x}}
\newcommand{\lp}{(\!(}
\newcommand{\rp}{)\!)}

\newcommand{\To}{\longrightarrow}

\newcommand\Hom{\operatorname{Hom}}
\newcommand\Spec{\operatorname{Spec}}
\newcommand\Sym{\operatorname{Sym}}
\newcommand\Res{\operatorname{Res}}

\newcommand\comp{\operatorname{comp}}
\newcommand\cone{\operatorname{cone}}

\begin{document}
\date{December 21, 2017}

\author{Francis Brown}
\address{All Souls College, Oxford, OX1 4AL, United Kingdom}
\email{Francis.Brown@all-souls.ox.ac.uk}

\author{Richard Hain}
\address{Department of Mathematics\\ Duke University, Box 90320\\
Durham, NC 27708 
}
\email{hain@math.duke.edu}

\thanks{The first author was partially supported by ERC grant 724638. The second was partially supported by National Science Foundation through grant DMS-1406420. He was also supported by ERC grant 724638 during a visit to Oxford during which this paper was written.}

\begin{title}[Algebraic de Rham theory for weakly holomorphic modular forms]{Algebraic de Rham theory for weakly holomorphic modular forms of level one}
\end{title}


\subjclass{Primary
 11F11, 
 11F23, 
 11F67; 
Secondary
11F25 
}

\begin{abstract}
We establish an Eichler--Shimura isomorphism  for weakly modular forms of level one. We do this by relating weakly modular forms with rational Fourier coefficients to the algebraic de~Rham cohomology of the modular curve with twisted coefficients. This leads to formulae for the periods and quasi-periods of modular forms.
\end{abstract} 

\maketitle

\section{Introduction}

Let $M_n$ denote the  $\Q$-vector space of modular forms of weight $n$ and level one with rational Fourier coefficients. Let $S_n \subset M_n$ denote the subspace of cusp forms.  The Eichler--Shimura isomorphism \cite{Eichler, Shimura} is usually expressed as a pair of isomorphisms 
\begin{align}
\label{intro: ES}
M_{n+2} \otimes_{\Q} \C   &\overset{\sim}{\To}  H^1(\SL_2(\Z); V_n^B)^+\otimes_{\Q}  \C   \\
{S}_{n+2} \otimes_{\Q} \C   &\overset{\sim}{\To}  H^1(\SL_2(\Z); V_n^B)^- \otimes_{\Q}  \C \nonumber
\end{align} 
where the right-hand side denotes group cohomology with coefficients in $V_n^B = \mathrm{Sym}^n V^B$, where $V^B$ denotes the standard two-dimensional representation of $\SL_2$ over $\Q$ with basis $\ba,\bb$, and  $\pm$ denote eigenspaces with respect to the real Frobenius (complex conjugation). One wants to think of this theorem  as a special case  of the comparison isomorphism between algebraic de Rham cohomology and Betti cohomology, each of which has a natural $\Q$-structure. The rational structures on these groups then enables one to define periods. Each cuspidal Hecke eigenspace, like an elliptic curve, should have four periods (two periods and two quasi-periods) corresponding to the entries of a $2\times 2$ period matrix. However, the   isomorphisms (\ref{intro: ES})  do not generate enough periods since each one only produces a  single period for every modular form. To obtain a full set of periods, one needs to consider ``modular forms of the second kind.''

In this note, we  compute the algebraic de Rham cohomology  of the moduli stack $\M_{1,1}$ of elliptic curves and relate it to weakly holomorphic modular forms (modular forms which are holomorphic on the upper half plane but with poles at the cusp). 
From this, we deduce a $\Q$-de Rham Eichler-Shimura isomorphism, and a definition of the period matrix of a Hecke eigenspace.

Before stating the main results, it may be instructive to review the familiar case of an elliptic curve $E$ over $\Q$ with equation $y^2=4 x^3-ux-v$. The de Rham cohomology $H^1_\dR(E,\Q)$ is a two-dimensional vector space over $\Q$, as is the Betti  (singular) cohomology $H^1(E(\C);\Q)$. The comparison isomorphism is a canonical isomorphism 
\begin{equation}
\label{CompforElliptic}
H^1_\dR(E;\Q) \otimes_{\Q} \C \overset{\sim}{\To} H^1(E(\C);\Q) \otimes_{\Q} \C \ .
\end{equation}
On the other hand,  the space $F^1 H_\dR^1(E;\Q) := H^0(E;\Omega^1_{E/\Q})$ is one-dimensional and spanned by the holomorphic differential $\frac{dx}{y}$. The Betti  cohomology $ H^1(E(\C);\Q) $ splits into two eigenspaces under the action of complex conjugation, with eigenvalues $\pm 1$. The analogue of the Eichler-Shimura isomorphisms, in this setting, are exactly the formulae
\begin{align*}
H^0(E;\Omega^1_{E/\Q})  \otimes_{\Q} \C  &\overset{\sim}{\To}  H^1(E(\C);\Q)^+ \otimes_{\Q} \C \\
H^0(E;\Omega^1_{E/\Q})  \otimes_{\Q} \C &\overset{\sim}{\To}  H^1(E(\C);\Q)^- \otimes_{\Q} \C
\end{align*} 
given by integrating the form $\frac{dx}{y}$ over invariant and anti-invariant cycles in $E(\C)$ with respect to complex conjugation, respectively. This is clearly a weaker statement than the comparison isomorphism (\ref{CompforElliptic}). To obtain all periods of the elliptic curve, one needs to consider, in addition, period integrals of the differential $\frac{xdx}{y}$ of the {\em second kind}, which provides an isomorphism
\begin{equation}
\label{eqn:alg-dr}
H^1_\dR(E;\Q) \cong \Q \frac{dx}{y} \oplus \Q\frac{xdx}{y}.
\end{equation}

\begin{rem}
There is also an isomorphism of
$$
H^1(E(\C);\C)/H^0(E;\Omega^1_{E/\Q})\otimes_{\Q}\C.
$$
with the space of anti-holomorphic differentials on $E(\C)$. Since this isomorphism is only  defined over $\C$,  one loses the rational structure, and cannot define periods in this manner.  The analogue for modular forms is the isomorphism
$$
(H^1(\SL_2(\Z);V_n^B)\otimes_{\Q}\C)/(M_{n+2}\otimes_{\Q}\C)
$$
with the space of antiholomorphic cusp forms of weight $n+2$.
\end{rem}

\subsection{Statement of the theorem}
Let $\mathcal{E}$ denote the universal elliptic curve over the moduli stack $\M_{1,1}$ (over $\Q$) of elliptic curves. It defines a rank two algebraic vector bundle $\VV$, equipped with the Gauss--Manin connection $\nabla$. For all $n\geq 1$, set 
$$
\VV_n = \mathrm{Sym}^{n} \VV
$$
and denote the induced connection by $\nabla$ also. Grothendieck  defined algebraic de Rham cohomology, which is a finite-dimensional $\Q$-vector space:
$$
H^1_\dR(\M_{1,1} ; \VV_n).
$$
In order to describe this space in terms of modular forms, for each $n\in \Z$, let $M_{n}^!$ denote the $\Q$-vector space of weakly holomorphic modular forms of weight $n$ that have a  Fourier expansion
$$
\sum_{n\geq -N} a_n q^n
$$
with $a_n\in \Q$. They can have negative weights. Such a form is called a cusp form if $a_0=0$. Let $S_n^! \subset M_n^!$ denote the subspace of cusp forms. Now consider the differential operator
\begin{equation}
\label{Ddefn}
\D= q\frac{d}{dq}\ .
\end{equation}
It does not in general preserve modularity, but an identity due to Bol \cite{bol} implies that its powers induce a  linear map
$$
\D^{n+1} : M_{-n}^! \To  S_{n+2}^!
$$
 for every $n\geq 0$. Our main theorem was inspired by the recent paper \cite{Guerzhoy}. After writing this note, we learnt that similar results for 
 modular curves of  higher level were implicitly obtained by Coleman \cite{coleman} in the $p$-adic setting, and independently by Scholl \cite{ScholldR}. As pointed out in the very recent paper \cite{SchollKazalicki}, a description of algebraic de Rham cohomology in terms of modular forms of the second kind seems not to have been  stated explicitly anywhere in the literature up until that point.  An approach using the Cousin resolution was subsequently given in \cite{candelori}. The following theorem can be indirectly deduced  from the results of these papers by viewing a modular form of level one as an invariant form of levels 3 and 4.

\begin{thm}
\label{thm: main} 
For each $n\ge 0$, there is a canonical isomorphism  of $\Q$-vector spaces
$$
\bw:  M_{n+2}^! / \D^{n+1} M^!_{-n}  \ \overset{\sim}{\To} \  H^1_\dR(\M_{1,1} ; \VV_{n})
$$
\end{thm} 

\noindent 
The space on the left contains the space of holomorphic modular forms as a subspace
$$
M_{n+2}  \ \subset \  M_{n+2}^! /\D^{n+1} M^!_{-n}\ .
$$
More precisely, the group $H^1_\dR(\M_{1,1}; \VV_n)$ carries a natural Hodge filtration
$$
H^1_\dR(\M_{1,1}; \VV_n) = F^0 \supset F^1 = \cdots = F^{n+1} \supset F^{n+2} = 0
$$
and $F^{n+1}$ is the image of  $M_{n+2}$ under $\bw$. That is,
$$
\bw: M_{n+2} \overset{\sim}{\To} F^{n+1} H^1_\dR (\M_{1,1} ; \VV_{n}).
$$
A splitting of the Hodge filtration is discussed in Section~\ref{sect: split}.

\subsection{Comparison isomorphism}
Grothendieck's algebraic de Rham theorem implies (cf.\ Section~\ref{sec:DR}) that there is a canonical isomorphism of complex vector spaces
$$
H^1_\dR(\M_{1,1} ; \VV_n) \otimes_{\Q} \C \overset{\sim}{\To}   H^1(\M_{1,1}(\C) ; \V^B_n)
$$
where the right-hand space is the Betti (singular) cohomology of $\M_{1,1}(\C)$ with coefficients in the complex local system $\V^B_n = \mathrm{Sym}^{n-1} R^1 \pi_* \C$, where $\pi: \mathcal{E}\rightarrow \M_{1,1}$ is the universal elliptic curve. Let $V_n^B$ denote its fiber $H^0(\HH,\V_n^B)$ at the tangent vector $\partial/\partial q$.  Since $\M_{1,1}(\C)$ is the orbifold quotient of the upper half plane by $\SL_2(\Z)$, its cohomology is computed by group cohomology of $\SL_2(\Z)$ and we immediately deduce the following consequence of Grothendieck's theorem:

\begin{cor}
There is a canonical isomorphism:
$$
\mathrm{comp}_{B,dR}  \ : \  H^1_\dR(\M_{1,1} ; \VV_{n}) \otimes_{\Q} \C \overset{\sim}{\To}   H^1(\SL_2(\Z); V_n^B\otimes_{\Q} \C)  \ .
$$ 
\end{cor}

Combined with the previous theorem, we deduce an algebraic de Rham version of the Eichler--Shimura isomorphism. It is the analogue for modular forms of the isomorphism (\ref{CompforElliptic}).
 
\begin{cor}
There is a canonical isomorphism
\begin{equation} 
\label{FullES}
M_{n+2}^! / \D^{n+1} M^!_{-n}   \otimes_{\Q} \C \quad  \overset{\sim}{\To}  \quad   H^1(\SL_2(\Z); V_n^B\otimes_{\Q} \C) \ .
\end{equation}
\end{cor}

The dimension of the space on the left-hand side was computed in \cite{Guerzhoy}, the dimension of the right-hand space by Eichler--Shimura: both are $1+ 2 \dim S_n$. Restricting the previous isomorphism to the subspace $M_{n+2}$ of holomorphic modular forms, and projecting onto the positive or negative eigenspaces with respect to complex conjugation on the right-hand space gives back the two isomorphisms (\ref{intro: ES}).

For any weakly holomorphic modular form $f\in M_{n+2}^!$, its image under the comparison isomorphism is given explicitly by the cohomology class of the cocycle:
\begin{equation}
\label{eqn:cocycle}
\gamma \mapsto (2\pi i)^{n+1} \int_{\gamma^{-1} z_0}^{z_0}  f(z) (z\ba-\bb)^{n} dz
\end{equation}
where $\ba,\bb$ is a basis of $V^B$, which we think of as the first rational Betti cohomology group of the elliptic curve $\C/(\Z\oplus z\Z)$. Its cohomology class does not depend on the choice of basepoint $z_0 \in \HH$. A different version of this map (and without the rational structures) was described in \cite{ESforMock}. See also \cite{BCD}, Theorem A.

\subsection{Periods}
The isomorphism (\ref{FullES}) is compatible with the action of Hecke operators.  The action of Hecke operators on the left-hand side was defined in \cite{Guerzhoy}. The eigenspace of an Eisenstein series is one-dimensional, that corresponding to a cusp form is two-dimensional. Let $f$ be a  cusp Hecke eigenform and $K_f \subset \R$ the field generated by its Fourier coefficients. Let
$$
V_f^\dR \ \subset \ \big( M_{n+2}^! / D^{n+1} M^!_{-n} \big) \otimes_{\Q} K_f
$$
denote the Hecke eigenspace generated by $f$. It is a two-dimensional $K_f$-vector space. Let
$$
V_f^B \  \subset  \  H^1(\SL_2(\Z); V^B_n\otimes_{\Q} K_f)
$$
denote the corresponding Betti eigenspace. It is also a two-dimensional $K_f$-vector space, and decomposes into invariant and anti-invariant eigenspaces with respect to the real Frobenius. We deduce from (\ref{FullES}) a canonical  isomorphism
$$
\comp_{B, \dR}: V_f^\dR \otimes_{K_f}  \C \overset{\sim}{\To} V_f^B\otimes_{K_f} \C.
$$

\begin{defn}
Define a \emph{period matrix} $P_f$ of $f$ to be the matrix of $\comp_{B, \dR}$ written in a $K_f$-basis of $V_f^\dR$ and $V_f^B$. We can assume that the basis of $V_f^B$ is compatible with decomposition into  eigenspaces for the action of the real Frobenius. It is of the form
$$
P_f =
\begin{pmatrix}
\eta_f^+  & \omega_f^+ \\ i \eta_f^-  & i \omega_f^-
\end{pmatrix} $$
where $\omega_f^{\pm}, \eta_f^{\pm} \in \R$.  It is well-defined up to right multiplication by a lower-triangular matrix with entries in $K_f$, and entries in $K_f^{\times}$ on the diagonal.  The $\omega^+_f,  i \omega_f^-$ are the holomorphic periods \cite{Manin}. 
\end{defn}

\begin{rem}
Only the holomorphic periods $\omega^+_f, i \omega_f^-$ can be obtained from the classical Eichler--Shimura isomorphisms (\ref{intro: ES}). 
\end{rem} 

\begin{thm}
\label{thm: determinant}
If $f$ has weight $2n$, then $\det (P_f)  \in (2 \pi i)^{2n-1} K^{\times}_f$.
\end{thm} 

\noindent\emph{Acknowledgments:} 
This paper is an outgrowth of a collaboration which originated in 2014--2015 during a stay at the Institute for Advanced Study. Both authors wish to thank the IAS for its hospitality and support. We would also like to thank Ma Luo for providing the isomorphism of $Z$ (defined in \S\ref{sec:Z}) with the Fermat cubic, and Jesse Silliman for bringing Coleman's work \cite{coleman} to our attention. Francis Brown also thanks the IHES for hospitality.

\section{Definitions/Background}

\subsection{Weakly holomorphic modular forms}
\begin{defn} For every $n \in \Z$, let $M_{n}^!$ denote the $\Q$-algebra of weakly holomorphic modular forms of level $1$ and weight $n$ with rational Fourier coefficients.  It is the $\Q$-vector space of holomorphic functions $f : \HH \rightarrow \C$ on the upper half plane  $\HH$ such that 
\begin{equation}
\label{intro:modcond}
f(\gamma z) = (cz +d)^{n} f(z) \quad \text{ for all } \quad  \gamma   =
\begin{pmatrix}
a & b\\c & d 
\end{pmatrix}  \in \SL_2(\Z) 
\end{equation} 
which admit a Fourier expansion of the form 
$$
f = \sum_{n\geq -N } a_n q^n, \qquad   a_n \in \Q \ .
$$
The space $S_{n}^! \subset M_n^!$ of cusp forms is the subspace of functions satisfying $a_0=0$. 
\end{defn} 

\begin{prop}[Bol's identity]
For all $n\geq 0$, there is a linear map 
$$
\D^{n+1} : M_{-n}^! \To M_{n+2}^!
$$
\end{prop}

\begin{proof}
This result follows automatically from our proof of Theorem~\ref{thm: main}. We  provide a more direct  proof for completeness.

For $f: \HH \rightarrow \C$ a real analytic modular form of weight $n\in \Z$ define 
$$
d_n f = \frac{1}{2\pi i}
\Big(
\frac{\partial f}{\partial z}  +  \frac{n f}{z - \overline{z}} \Big) \ .
$$
It is well-known by Maass, and easily verified, that this operator respects the transformation property (\ref{intro:modcond}), and therefore $d_nf$ is a real analytic modular form of weight $n+2$. 

The proposition follows from the following identity, for all $n\geq 0$, 
$$
\D^{n+1} = d_n d_{n-2} \ldots d_{2-n} d_{-n}\ .
$$
To verify this, write a real analytic function  on $\HH$ as a formal power series in $(z- \overline{z})$ and $\overline{z}$. It suffices to verify the formula for 
$$
f_{a,b} = (2\pi i (z - \overline{z}) )^a (- 2  \pi i  \overline{z})^b\ , 
$$
where $a,b\geq 0$. 
We check that $d_{m} f_{a,b} =   (a+m) f_{a-1,b}$, and hence
$$
d_n d_{n-2} \dots d_{2-n} d_{-n}  \big( f_{a,b} \big)=  a(a-1) \dots (a-n) f_{a-n-1,b}\ .
$$
On the other hand, $\log q = 2  \pi i  z$ and $\log \overline{q} = - 2  \pi i \overline{z}$ and hence 
$$
f_{a,b} = (\log q + \log \overline{q})^a ( \log \overline{q})^b \ .
$$
Since $\D = q\partial/\partial q = {\partial / \partial ( \log q)}$, the operator $\D^{n+1} $ acts on  $f_{a,b}$ in an identical manner.  
\end{proof}

\subsection{Moduli of elliptic curves}
\label{sec:moduli}

Let $\kk$ be a commutative ring with $6\in \kk^\times$. In much of this paper, $\kk$ will be either the universal such ring, $\kk=\Z[\frac{1}{6}]$, or $\kk=\Q$. The formula
$$
\lambda \cdot u = \lambda^4 u ,
\qquad \lambda \cdot v =  \lambda^6 v
$$
defines a left action
\begin{equation}
\label{leftGmaction}
\mu : \G_m \times \A^2  \To \A^2
\end{equation}
of the multiplicative group $\Gm$ on the affine plane $\A^2 := \Spec \kk[u,v]$, and  defines a grading on $\kk [u,v]
$ called the weight.  The discriminant function
$$
\Delta = u^3 - 27 v^2
$$
has weight 12 under this action, so that $\Gm$ also acts on the graded ring $\kk[u,v][\Delta^{-1}]$. Let $D$ be the vanishing locus of the discriminant $\Delta$. Set
$$
X = \A^2 - D := \Spec \kk[u,v][\Delta^{-1}] .
$$
Then $X$ is the moduli scheme of elliptic curves over $\kk$ together with a non-zero abelian differential, \cite[\S (2.2.6)]{katz-mazur}. Its coordinate ring $\Or(X)$
$$
\Or(X) = \kk[u,v][\Delta^{-1}] = \bigoplus_{m \text{ even}} \gr_m\Or(X)
$$
is graded by the $\Gm$ action. The sheaf of regular functions on $X$ will be denoted by $\Or_X$.

The universal elliptic curve $\E$ over $X$ is the subvariety of $\P^2\times_{\kk} X$ which is the Zariski closure of the affine scheme defined  by the equation
$$
y^2 = 4x^3 - ux -v \in \Or(X)[x,y].
$$
The universal abelian differential  on $\E$ is $dx/y$. The multiplicative group $\G_m$ acts on $(x,y)$ by $\lambda \cdot(x,y) = (\lambda^2 x, \lambda^3 y)$ and on the abelian differential by $\lambda\cdot \frac{dx}{y} = \lambda^{-1}\frac{dx}{y}$.

The moduli stack of elliptic curves over $\kk$ is the stack quotient
$$
\M_{1,1}/\kk = \Gm \bbs X
$$
of $X$ by $\Gm$. Its Deligne--Mumford compactification is
$$
\Mbar_{1,1}/\kk = \Gm \bbs Y,
$$
where $Y = \A^2-\{0\}$. In down to earth terms, to work on $\M_{1,1}/\kk$ is to work $\Gm$-equivariantly on $X$, and to work on $\Mbar_{1,1}/\kk$ is to work $\Gm$-equivariantly on $Y$.

\subsection{Upper half plane description} \label{sect: UHplane}

When $\kk=\C$, this description of $\M_{1,1}$ relates to the traditional upper half plane model via Eisenstein series. Denote by $\GE_{2n}$ the normalized Eisenstein series
$$
\GE_{2n}(q) =
-\frac{B_{2n}}{4n} + \sum_{m=1}^\infty \sigma_{2n-1}(m)q^m
$$
of weight $2n$ where $B_{k}$ is the $k$th Bernoulli number, $q = e^{ 2 \pi i z}$, and $\sigma_k(m)=\sum_{d|m} d^k$ the divisor function.

Define a map $\rho : \HH \to X(\C)$ by $\rho(z)=(u(z),v(z))$, where
\begin{equation}
\label{eqn:projn}
u = 20\GE_4(z) \quad \text{ and }  \quad v= \textstyle{\frac{7}{3}}\GE_6(z).
\end{equation}
The map $\rho$ factors through the punctured $q$-disk and induces a graded ring isomorphism to the space of holomorphic modular forms
$$
\rho^\ast : \Q[u,v] \overset{\simeq}{\To} \bigoplus_{n\text{ even}} M_n.
$$
The pull-back of the discriminant $\Delta$ is the Ramanujan $\tau$-function $\Delta(z)$, which vanishes nowhere on $\HH$. It follows that $\rho$ induces a graded ring isomorphism
$$
\rho^\ast : \Or(X) \overset{\simeq}{\To} \bigoplus_{n\text{ even}} M_n^! \ .
$$
and that the image of $\rho$ is indeed contained in $X(\C)$.  The ring of weakly modular forms is therefore nothing other than the affine ring of functions on $X$. 

The elliptic curve $\C/(\Z \oplus z\Z)$ is mapped isomorphically to the elliptic curve
$$
y^2 = 4 x^3 - 20\GE_4(z) - \frac{7}{3}\GE_6(z)
$$
by $w \mapsto (\wp_z(w)/(2\pi i)^2,\wp_z'(w)/(2\pi i)^3)$, where $\wp_z(w)$ denotes the Weierstrass $\wp$-function. The abelian differential $dx/y$ pulls back to $2\pi i\, dw$. This implies that the map (\ref{eqn:projn}) induces an isomorphism
$$
\SL_2(\Z)\bbs \HH \overset{\simeq}{\To} \big(\Gm\bbs X\big)(\C) = \M_{1,1}(\C)
$$
of analytic stacks, since $\rho(\gamma z) = ((cz+d)^4 u(z), (cz+d)^6 v(z))$ for $\gamma \in \SL_2(\Z)$ of the form (\ref{intro:modcond}),   by modularity of $\GE_4$ and $\GE_6$.

The following lemma motivates the choice of the basis of one-forms considered later.

\begin{lem}
\label{lem: dq/q}
Pulling back along (\ref{eqn:projn}) we have:
$$
\frac{ 2u dv - 3v du}{\Delta} = \frac{2}{3}\frac{dq}{q}\ .
$$
\end{lem}

\begin{proof}  
Formulae due to Ramanujan imply that 
$$
\D E_2 = (E_2^2 - E_4)/12, \quad  \D E_4  = (E_2 E_4 - E_6)/3, \quad \D E_6 = (E_2E_6-E_4^2)/2\ ,
$$
where $E_{2n}= - \frac{4n}{B_{2n}} \GE_{2n} $ are the Eisenstein series normalised such that their constant Fourier coefficient is $1$. 
These are easily verified by computing the first few terms in their Fourier expansion \cite[Prop.~15]{Zag123}. It follows that 
$$
3 E_6 \,\D E_4 - 2 E_4 \,\D E_6 =  E_4^3 - E_6^2 = 1728 \Delta\ .
$$
Now substitute $E_4 = 240 \GE_4 = 12 u$ and $E_6 = -504 \GE_6= -216 v$.  
\end{proof}

\begin{rem}
\label{UVexp}
The functions $u,v$ have the following $q$-expansions:
\begin{align*}
u  =  20\, \GE_4  &= \frac{1}{12}+20 q+180 q^2+560 q^3+1460 q^4+ \cdots  \nonumber \\
v  =  \frac{7}{3}\, \GE_6 &= -\frac{1}{216}+ \frac{7}{3}  q+77 q^2+\frac{1708}{3}q^3+\frac{7399}{3} q^4+ \cdots
\end{align*} 
They have coefficients in $\Z[\textstyle{\frac{1}{6}}]$.  Furthermore, 
$$
\Delta^{-1} = \frac{1}{q} + 24+324 q+3200 q^2+25650 q^3+176256 q^4 +\cdots
$$
has integer coefficients. It follows that 
$$
\frac{M_{n+2}^!}{\D^{n+1} M_{-n}^!}
$$
has a natural $\Z[\textstyle{\frac{1}{6}}]$-structure, given by series whose Fourier coefficients lie in $\Z[\frac{1}{6}]$. This is because the operator $\D$ acts on the ring $\Z[[q]]$.
\end{rem}

The $\Z[\textstyle{\frac{1}{6}}]$ structure on the affine ring of $X/\Z[\frac{1}{6}]$ coincides, by Remark~\ref{UVexp}, with the $\Z[\frac{1}{6}]$ structure on Fourier expansions.

\subsection{Differential forms} \label{sect: Diffforms}
The following one-forms play a special role:
\begin{equation}
\label{def: forms}
\psi =  \frac{1}{12} \frac{d \Delta}{\Delta}   \quad , \quad \omega = \frac{3}{2}  \Big(\frac{2 u\, dv - 3 v \, du}{\Delta} \Big)  \qquad \in \quad  \Omega^1(X) \ .
\end{equation} 
They have weights $0$ and $-2$, respectively. By Lemma~\ref{lem: dq/q},  we have
$$
\rho^\ast \omega =  \frac{dq}{q} \ .
$$
Both $\omega$ and $\psi$ have logarithmic singularities along the discriminant locus $D$.

\begin{lem}
If $h \in \gr_{m} \Or(X)$, then $\Delta \, \psi \wedge \omega = \frac{3}{4}  \,du \wedge dv$ and 
\begin{align}
\label{dlogwedgeomega}
\frac{dh}{h} \wedge \omega & =  m  \,  \psi \wedge \omega\ ,   \\
 d (h \omega) & =  \big(m -2 \big)\,   h \, \psi \wedge \omega \  . 
\nonumber 
\end{align}
\end{lem} 

\begin{proof}
Since $d\log (h_1h_2) = d\log(h_1) + d\log(h_2)$, it suffices to verify the first equation for $h=u$ and $h=\Delta$. For the second equation, use $d \omega = -2  \, \psi \wedge \omega$  to write
$$
d(h\omega) = h \frac{dh}{h} \wedge \omega -2 \, h \,  \psi \wedge \omega
$$
and conclude using the first equation. 
\end{proof} 

In particular, since $\psi$ and $\omega$ are pointwise linearly independent, we have
$$
\Omega^1(X) =
\Or(X) du \oplus \Or(X) dv
\cong  \Or(X) \psi \oplus \Or(X) \omega\ .
$$

\begin{cor}
For every $ f\in \gr_m \Or(X)$ we have 
$$
df  =  \vartheta(f) \, \omega +   m  f\,  \psi\ ,
$$
where $\vartheta:\Or(X) \rightarrow \Or(X)$ is the  derivation  of weight two  defined by 
\begin{equation}
\label{thetadef}
\vartheta  =   6 v \frac{\partial}{\partial u}   +  \frac{u^2}{3} \frac{\partial}{\partial v}  .
\end{equation} 
\end{cor} 

\begin{proof}
There exist  unique $f_0,f_1 \in \Or(X)$ such that 
$$
df = f_0 \,  \omega  +  f_1  \psi \ .
$$
Use (\ref{dlogwedgeomega}) to deduce that $df \wedge \omega = m f \psi \wedge \omega$, which yields $f_1 = m f$.  The linear map $\vartheta: f \mapsto f_0$ is a derivation which necessarily satisfies $\vartheta(\Delta)=0$.  This is certainly true of the formula (\ref{thetadef}). It therefore suffices to verify that  $\vartheta(u)=6v$. For this,  use the fact that $d \Delta = 3u^2 du-54 v dv$ to deduce that $du \wedge \psi = \frac{-9 v}{2} \frac{du \wedge dv}{\Delta}$. Comparing with $\omega \wedge \psi = -\frac{3}{4} \frac{du \wedge dv}{\Delta}$ implies that $\vartheta(u) = 6v$ as required.
 \end{proof} 

\begin{rem}
The derivation $\vartheta$ is closely related to the Serre derivative \cite[(53)]{Zag123}.
\end{rem}

Consider the pull-back along (\ref{leftGmaction}) followed by the natural  map
$$
\Omega^1_X \overset{\mu^\ast}{\To} \Omega^1_{\G_m \times X} \To \Omega^1_{(\G_m \times X)/X}\
$$
to relative K\"ahler differentials. Taking global sections gives a natural $\Or(X)$-linear map 
$$
\pi^\ast : \Omega^1(X) \To \Or(X) \otimes_{\kk} \Omega^1(\G_m) = \Or(X)[\lambda^{\pm}]\, d\log \lambda  \ ,
$$
where $\lambda$ is the coordinate on $\G_m$.

Say that an element of $\Omega^1(X)$ is {\em proportional} to $\omega$ if it lies in the subspace $\Or(X) \omega$ of $\Omega^1(X)$ spanned by $\omega$.

\begin{lem}
\label{lem: omegacriterion}
A form $\eta \in \Omega^1(X)$ is proportional to $\omega$ if and only if $\pi^\ast \eta=0$. Furthermore,  $\pi^\ast \psi = d \log \lambda$.
\end{lem}

\begin{proof}
Via the natural isomorphism of $\Or(X)[\lambda^{\pm}]$-modules:
$$
\big(\Omega^1(\G_m) \otimes_{\kk}  \Or(X)\big) \oplus \big(\Omega^1(X) \otimes_{\kk} \kk[\lambda^{\pm}]\big)\To \Omega^1(\G_m \times X)
$$
we can compute
\begin{equation}
\label{mstar}
\mu^\ast(\omega ) =   (0, \lambda^{-2} \omega)  \text{ and }
\mu^\ast(\psi) = (d\log \lambda , \psi).
\end{equation} 
This follows from calculating
$$
\mu^\ast (2 udv - 3 v du) = 2 \lambda^4 u d(\lambda^6 v) - 3 \lambda^6 v d (\lambda^4 u) = \lambda^{10}\omega
$$
and noting that the terms involving $d \lambda$ cancel. More generally, we verify that for any $h \in \gr_m \Or(X)$, we have
$$
\mu^\ast dh = d( \lambda^{m} h)
=  \lambda^{m} dh + m \lambda^{m-1} h d\lambda
$$
and therefore $\mu^\ast  d\log h  = (m \, d \log \lambda, d\log h)$. Setting $h = \Delta$ proves (\ref{mstar}), from which the lemma immediately follows.
\end{proof}

\subsection{Alternative description of $\M_{1,1/\Q}$}
\label{sec:Z}

A complementary approach to constructing $\M_{1,1}$ as a stack over $\kk$ is as a quotient of the affine subscheme $Z$ of $X$ defined by $\Delta=1$. Its affine ring $\Or(Z)$ is $\kk[u,v]/I$, where $I$ is the graded ideal of $\kk[u,v]$ generated by $\Delta-1$, where $\kk$ is any commutative ring with $6 \in \kk^{\times}$.  Since $\Delta$ has weight $12$, the affine group scheme
$$
\mmu_{12} = \mathrm{Spec}\, \kk[\lambda]/(\lambda^{12}-1),
$$
with $\lambda$ group-like, acts on $Z$.

\begin{rem}
The affine scheme $Z$ is isomorphic, over $\kk$, to the Fermat cubic minus its identity element. In fact, the closure of the locus $u^3-27v^2 = 1/4$ is isomorphic to the Fermat cubic $x^3+y^3=1$, where $u=1/(x+y)$ and $v=(x-y)/6(x+y)$. \end{rem}

The inclusion $\mmu_{12} \hookrightarrow \Gm$ induces an isomorphism of  (the constant group scheme) $\Z/12\Z$ with the character group of $\mmu_{12}$. Denote the congruence class of $n \bmod 12$ by $[n]$. The $\mmu_{12}$ action on $Z$ gives a $\Z/12\Z$-grading of its coordinate ring:
$$
\Or(Z) = \bigoplus_{n\bmod 12} \Or(Z)_{[n]}.
$$
If $n$ is odd then $\Or(Z)_{[n]}=0$. Since the inclusion $j : Z \hookrightarrow X$ is $\mmu_{12}$ equivariant, the restriction homomorphism induces a ring homomorphism
\begin{equation}
\label{eqn:ring_homom}
j^\ast : \Or(X)^\Gm \to \Or(Z)^{\mmu_{12}}
\end{equation}
and, for each $m\in \Z$, a homomorphism
\begin{equation}
\label{eqn:mod_homom}
j^\ast : \gr_m \Or(X) \To \Or(Z)_{[m]}
\end{equation}
of modules over (\ref{eqn:ring_homom}). 

\begin{lem}
\label{lem:Zisom}
The homomorphism (\ref{eqn:ring_homom}) is an isomorphism, so that
$$
j : \mmu_{12}\bbs Z \overset{\simeq}{\To} \G_m\bbs X
$$
is an isomorphism of stacks. Moreover, for each $m \in \Z$, (\ref{eqn:mod_homom}) is an isomorphism.
\end{lem}
 
\begin{proof}
Since $\Delta-1$ is $\mmu_{12}$-invariant, it suffices for the first part to show that 
$$
\Or(X)^{\G_m} = \kk[u,v,\Delta^{-1}]^{\G_m}
\To \kk[\overline{u},\overline{v}]^{\mmu_{12}}/I^{\mmu_{12}}
= \Or(Z)^{\mmu_{12}}
$$
is an isomorphism, where $\overline{u} =j^{\ast}(u)$ and $\overline{v}= j^{\ast}(v)$ are the images of $u, v$.  Since $\overline{u}$ has weight 4 and $\overline{v}$ has weight 6, it follows that $\kk[\overline{u},\overline{v}]^{\mmu_{12}}= \kk[\overline{u}^3,\overline{v}^2]$.
The inverse is induced by the map 
$$
\kk[\overline{u}^3,\overline{v}^2] \To   \kk[u,v,\Delta^{-1}]^{\G_m}
$$
which sends $\overline{u}^3$ to $u^3\Delta^{-1}$ and $\overline{v}^2$ to $v^2 \Delta^{-1}$. It vanishes on $I^{\mmu_{12}}$. This proves that (\ref{eqn:ring_homom}) is an isomorphism.

For all integers $k$, multiplication by $\Delta^k$ gives an isomorphism 
$$
\gr_m \Or(X) \overset{\simeq}{\To} \gr_{m+12k} \Or(X) 
$$
of $\gr_0\Or(X) = \Or(X)^{\G_m}$-modules. 
It therefore suffices to prove that (\ref{eqn:mod_homom}) is an isomorphism  for $m=0,4,6,8,10,14$, which form a complete set of representatives for even numbers modulo $12$. But $\Or(Z)$ is isomorphic   to the free $\Z/12\Z$-graded $\Or(Z)^{\mmu_{12}} \cong \kk[\overline{u}^3 , \overline{v}^2]$-module generated by monomials in $\overline{u} $ and $\overline{v}$ that are of degree $< 3$ in $\overline{u}$ and degree $< 2$ in $\overline{v}$, where $\overline{u}^3-27 \overline{v}^2=1$.  These are
$$
1, \overline{u}, \overline{v}, \overline{u}^2, \overline{u}\overline{v} , \overline{u}^2 \overline{v},
$$
and have weights $0, 4,6,8,10,14$, respectively.  Similarly, 
$\gr_m \Or(X)$ is isomorphic to the free  graded
 $ \Or(X)^{\G_m} \cong \kk[u^3 \Delta^{-1}, v^2 \Delta^{-1} ] $
module generated by the monomials 
$
1,  u ,  v ,  u^2,  u v , u^2v.
$
\end{proof}

\subsection{Gauss--Manin connection}
\label{sect: GM} (\cite{Katz} \S A1).
The vector bundle $\VV$ over $X$ is defined to be the restriction of the trivial rank 2 vector bundle
$$
\VVbar := \Or_Y\Sn \oplus \Or_Y\Tn
$$
on $Y:=\A^2-\{0\}$ to $X$. The multiplicative group acts on it by
$$
\lambda \cdot \Sn = \lambda \Sn, \qquad \lambda \cdot \Tn = \lambda^{-1} \Tn.
$$
So $\VVbar$ can be regarded as a  vector bundle on the moduli stack $\Mbar_{1,1}$ and $\VV$ as a vector bundle on $\M_{1,1}$.
Set $\VV_n = \Sym^n \VV$ for all $n \geq 1$.

The connection on $\VVbar$, and its symmetric powers
$$
\VVbar_n := \Sym^n \VVbar = \bigoplus_{j+k=n}\Or_Y \Sn^j\Tn^k
$$
is defined by
\begin{equation}
\label{nabladef}
\nabla = d + 
\begin{pmatrix}
\Sn \ & \Tn
\end{pmatrix}
\begin{pmatrix}
\psi  &  \omega \\
-\frac{u}{12}\,  \omega  &  -\psi
\end{pmatrix}
\begin{pmatrix}
\frac{\partial}{\partial \Sn}  \\  \frac{\partial}{\partial \Tn}
\end{pmatrix}  \ .
\end{equation}
It is $\Gm$-invariant, and thus defines a rational connection on $\VVbar_n$ with regular singularities and nilpotent residue along the discriminant divisor $D$.

Set $\V = R^1 \pi_\ast \kk_X$ where $\pi : \E \to X$ is the universal elliptic curve. It is proved in \cite[Prop.~19.6]{KZB} that when $\kk\subset \C$ there is a natural isomorphism
$$
\V^B\otimes_{\kk} \Or_X^\an \cong \VV^\an\otimes_{\kk} \C
$$
of bundles with connection over $X(\C)$, where the left hand bundle is endowed with the Gauss--Manin connection. Under this isomorphism $\Tn$ corresponds to the section $dx/y$ of $(R^1\pi_\ast \C_X)\otimes_{\C} \Or_X$ and $\Sn$ to the section $xdx/y$. For later use, we set
\begin{equation}
\label{eqn:Vn}
\V^B_n = \Sym^n \V^B.
\end{equation}

When pulled back to the upper half plane (and $q$-disk), the connection can also be written down in terms of the frame $\An$ and $\Tn$, where $\An$ is the section corresponding to the Poincar\'e dual of the element $\ba$ of $H_1(\Z/(\Z+z\Z))$ corresponding to the curve from 0 to 1. The two framings are related by\footnote{
This follows from the formulas in \cite[\S 19.3]{KZB}: $\Tn$ and $\Sn$ are obtained from $\hat{T}$ and $\hat{S}$ there by setting $\xi = (2\pi i)^{-1}$ and taking $A$ to be $\ba$.}
\begin{equation}
\label{STtoAT}
\Sn= \An + 2\, \GE_2(q) \Tn.
\end{equation}
In this frame, the Gauss--Manin connection is given by
\begin{equation}
\label{nablaAT}
\nabla = d + \An \frac{\partial}{\partial \Tn}\otimes\frac{dq}{q}
= 2\pi i\big(\D + \An \frac{\partial}{\partial \Tn}\big)\otimes dz,
\end{equation}
where, as above, $\D$ was the differential operator $q\partial/\partial q = (2\pi i)^{-1}\partial/\partial z$.

\begin{lem}
\label{lem: nablafAT}
If $f = \sum_{j+k=n} f^{j,k} \An^j \Tn^k$  is a real analytic section of $\VV_n$, then
$$
\nabla f = 2\pi i
\sum_{j+k=n}  \big(\D f^{j,k} +  (k+1)f^{j-1,k+1}\big)  dz \otimes \An^j \Tn^k,
$$
where we define  $f^{n+1,-1}=f^{-1,n+1} = 0$.
\end{lem} 
\begin{proof}
Apply the connection (\ref{nablaAT}) to the section $f$. \end{proof}

Pulling back via the map $\pi^\ast$ defines  a relative connection on $\G_m$ over $\Or_X$, which by (\ref{mstar}) is of the form 
\begin{equation}
\label{nablaaslambda}
\pi^\ast( \nabla) = d +
\begin{pmatrix}
\Sn \ & \Tn
\end{pmatrix}
\begin{pmatrix}
\frac{d \lambda}{\lambda } & 0 \\   0  &  - \frac{d\lambda}{\lambda}
\end{pmatrix}
\begin{pmatrix}
\frac{\partial}{\partial \Sn}  \\ \frac{\partial}{\partial \Tn}
\end{pmatrix}  
\end{equation}
 A section of this relative connection is flat if and only if it is $\G_m$-equivariant.

\section{De~Rham cohomology}
\label{sec:DR}

In this section, we take $\kk=\Q$. Since $X$ is affine, it follows from the version of Grothendieck's algebraic de Rham theorem with coefficients in a connection \cite[Cor.~6.3]{deligne} that for each $n\ge 0$, $H_\dR^1(X,\VV_n)$ is computed by the complex
$$
\Omega^\bdot(X, \VV_n) := \big[\Gamma (X,\VV_n)  \overset{\nabla}{\To}  \Gamma (X,\Omega_X^1\otimes_{\Or_X} \VV_n) \overset{\nabla}{\To}  \Gamma (X,\Omega_X^2 \otimes_{\Or_X} \VV_n) \big].
$$
Since $\nabla$ is equivariant with respect to the action of $\G_m$, we obtain the subcomplex
\begin{equation}
\label{Gmcomplex}
\Omega^\bdot(X, \VV_n)^\Gm = 
\big[\Gamma (X, \VV_n)^\Gm  \overset{\nabla}{\To}  \Gamma(X,\Omega_X^1\otimes_{\Or_X} \VV_n)^\Gm \overset{\nabla}{\To}  \Gamma (X,\Omega_X^2 \otimes_{\Or_X} \VV_n)^\Gm\big]  
\end{equation}
of invariant forms. Since $\Gm$ is connected, this also computes the de~Rham cohomology of $X$ with coefficients in $\VV_n$. The Leray spectral sequence for the $\Gm$-bundle $p : X \to \M_{1,1}$
$$
H^j(\M_{1,1},R^k p_* \VV_n) \quad \To \quad H^{j+k}(X,\VV_n)
$$
has only two non-vanishing rows, which implies that there is an exact sequence
\begin{equation}
\label{eqn:ses}
0 \To H_\dR^j(\M_{1,1},\VV_n) \overset{p^\ast}{\To} H_\dR^j(X,\VV_n) \To\  H_\dR^{j-1}(\M_{1,1},R^1 p_*\VV_n) \To 0 
\end{equation}
for all $j\geq 0$,  where  we recall that  $\VV_n$  denotes both the bundle  $p_* \VV_n$ on $\M_{1,1}$ and $\VV_n$ on $X$. 
Since $\VV_n$ is trivial on the $\Gm$ orbits on $X$,
$$
H_\dR^{j-1}(\M_{1,1},R^1 p_*\VV_n) \cong H_\dR^1(\G_m ;\Q) \otimes_{\Q}  H_\dR^{j-1}(\M_{1,1}, \VV_n).
$$
If $n>0$, then $H^j(\M_{1,1},\VV_n)$ vanishes when $j\neq  1$. 
We deduce natural isomorphisms
$$
p^\ast : H_\dR^1(\M_{1,1},\VV_n) \overset{\simeq}{\To} H_\dR^1(X,\VV_n)
$$
for all $n>0$ and 
$$
\pi^\ast: H_\dR^1 (X; \VV_0)   \overset{\simeq}{\To} H_\dR^1(\G_m ;\Q)\ .
$$
By computations in \S\ref{sect: Diffforms}, the left-hand side is  generated by $[\psi]$.

\begin{rem}
Since $\M_{1,1} = \mmu_{12}\bbs Z$, the homology of the complex $\Omega^\bdot(Z,j^\ast\VV_n)^{\mmu_{12}}$ is $H^\bdot_\dR(\M_{1,1},\VV_n)$. Below (cf.\ Prop.~\ref{prop:quism}) we show that there is a canonical isomorphism
$$
\Omega^\bdot(X,\VV_n)^\Gm \cong \Omega^\bdot(Z,j^\ast\VV_n)^{\mmu_{12}} \oplus \Omega^\bdot(Z,j^\ast\VV_n)^{\mmu_{12}}[-1]
$$
where the shift is given by multiplication by $\psi$. This gives a natural splitting of  (\ref{eqn:ses}), where the shift $[-1]$ is given by cup product with $[\psi]$: 
\begin{equation}
\label{eqn:coho_splitting}
H_{\dR}^\bdot(X,\VV_n) \cong H_{\dR}^\bdot(\M_{1,1},\VV_n) \oplus H_{\dR}^\bdot(\M_{1,1},\VV_n)[-1]
\end{equation}
\end{rem} 

\subsection{The subcomplex $\Omega^\bdot(X,\VV_n)^\omega$}

Recall that we can identify $\gr_n \Or(X)$ with $M_n^!$ via the isomorphism 
$$
\rho^\ast : \gr_n \Or(X)
= \gr_n  \Q [u,v][\Delta^{-1}] \overset{\sim}{\To} M_{n}^!
$$
since every holomorphic modular form can be uniquely written as a polynomial in the Eisenstein series $u$ and $v$.

\begin{defn}
For each $n\geq 0$, define 
\begin{equation}
\label{omegadef}
\bw :  M_{n+2}^!  \To \Gamma(X, \Omega^1_X\otimes_{\Or_X} \VV_{n})^\Gm
\end{equation}
to be the $\Q$-linear map that takes $f$ to $\omega_f :=  f  \omega \, \Tn^{n}$ where $\omega$ and $\Tn$ were defined in (\ref{def: forms}) and \S\ref{sect: GM}.
\end{defn} 

The first task in proving Theorem~\ref{thm: main} is to show that $\bw$ induces a linear map 
\begin{equation}
\label{omegamap} 
\bw :  M_{n+2}^! / D^{n+1} M_{-n}^!  \To  H_\dR^1(\M_{1,1}; \VV_{n})\ .
\end{equation}
Here we take the first steps in this direction. The following lemma implies that $\omega_f$ is $\nabla$-closed.
 
\begin{lem}
\label{lem:omega_closed}
For all $j,k\geq 0 $ and $f  \in M_{k-j+2}^!$, the one-form
$$
\eta =   f \omega  \, \Sn^j \Tn^k \quad \in \quad \Gamma(X, \omega \otimes \VV_{j+k})
$$
is $\G_m$-invariant and satisfies $\nabla \eta=0$. Consequently,
\begin{equation}
\label{omegasectionisclosed}
\Gamma (X, \omega\otimes\VV_{j+k})^\Gm \subset \ker \nabla \ .
\end{equation}
\end{lem}

\begin{proof} We have $ f\in \gr_{k-j+2} \Or(X)$. By the Leibniz rule
\begin{equation}
\label{inproof: nablaomegaf}
\nabla \eta =  d \big(f \omega \big)\, \Sn^j \Tn^k -  f \omega  \wedge \nabla \Sn^j \Tn^k \ .
\end{equation} 
From the connection formula (\ref{nabladef}), we have 
$$
\nabla  \Sn^j \Tn^k
\equiv (j-k) \,  \psi \,  \Sn^j \Tn^k  \pmod {\omega \Or(X)}
$$ 
and therefore 
$$
f \omega \wedge \nabla \Sn^j  \Tn^k
= (k-j) f  \psi \wedge \omega \, \Sn^j \Tn^k\ .
$$
By the second equation of (\ref{dlogwedgeomega}), we find that 
$$
d (f \omega )
= (k-j+2-2)  f\, \psi \wedge \omega
= (k-j)  f\, \psi \wedge \omega\ .
$$
Substituting the two previous expressions into (\ref{inproof: nablaomegaf}) implies that $\nabla \eta=0$. 
\end{proof}

\begin{lem}
\label{lem: nologD}
The image of $\nabla : \Gamma(X,\VV_n)^\Gm \To \Gamma(X, \Omega^1_X \otimes \VV_n)^\Gm$ lies in the subspace $\Gamma(X,\omega\otimes\VV_n)^\Gm$ of forms proportional to $\omega$.
\end{lem}

\begin{proof}
It suffices, by Lemma~\ref{lem: omegacriterion}, to show that if 
$$
f = \sum_{j+k=n} f^{j,k} \Sn^j \Tn^k
$$
is $\G_m$-equivariant, then $\pi^\ast (\nabla f )= 0$. It follows from (\ref{nablaaslambda}) that  
$$
\pi^\ast (\nabla f) =
\sum_{j+k=n} \Big(\pi^\ast  d f^{j,k} + (j-k)  \lambda^{k-j} f^{j,k}\frac{d\lambda}{\lambda }\Big) \Sn^j \Tn^k \ .
$$
Each term in brackets vanishes, since it expresses the fact that $\lambda \cdot f^{j,k} = \lambda^{k-j} f^{j,k}$, i.e., $f^{j,k} \in \gr_{k-j} \Or$,  which is equivalent to the $\G_m$-equivariance of $f$ since $\Sn$ has weight $+1$ and $\Tn$ has weight $-1$. 
\end{proof}

This lemma implies that $\nabla$ acts on $\Gamma(X,\VV_n)^\Gm $ via the connection 
\begin{equation}
\label{nablaomega}
\nabla^\omega  = \omega   \, \vartheta +
\begin{pmatrix}
\Sn \ & \Tn
\end{pmatrix}
\begin{pmatrix}
0 &   \omega  \\  - \frac{u}{12}  \omega   &  0
\end{pmatrix}
\begin{pmatrix}
\frac{\partial}{\partial \Sn}  \\  \frac{\partial}{\partial \Tn}
\end{pmatrix}
\end{equation}
where $\vartheta$ is the operator (\ref{thetadef}).

The following is an immediate consequence of Lemmas~\ref{lem:omega_closed} and \ref{lem: nologD}.

\begin{cor}
\label{cor:sub-cplex}
For all $n\ge 0$,
\begin{equation}
\label{eqn:complex}
\Omega^{\bdot}(X,\VV_n)^{\omega} :  = 
\xymatrix{[\Gamma(X,\VV_n)^\Gm \ar[r]^(.45){\nabla^\omega} & \Gamma(X,\omega\otimes\VV_n)^\Gm \ar[r] & 0]
}
\end{equation}
is a subcomplex of $\Omega^\bdot(X,\VV_n)^\Gm$.
\end{cor}

The kernel of the restriction mapping
$$
j^\ast : \Omega^\bdot(X,\VV_n)^\Gm
\To \Omega^\bdot(Z,j^\ast\VV_n)^{\mmu_{12}}.
$$
contains the ideal generated by $\psi$. It therefore induces a homomorphism
\begin{equation}
\label{eqn:restn}
j^\ast : \Omega^\bdot(X,\VV_n)^\omega
\To \Omega^\bdot(Z,j^\ast\VV_n)^{\mmu_{12}}.
\end{equation}

\begin{lem}
\label{lem:Zrestn}
For all $n\ge 0$, the restriction map (\ref{eqn:restn}) is an isomorphism.
\end{lem}

\begin{proof}
Both complexes have length 2. That $j^\ast$ is an isomorphism in degree 0 follows directly from Lemma~\ref{lem:Zisom}. To prove the assertion in degree 1, note that elements of $\Gamma(X,\omega \otimes \VV_n)^\Gm$ are of the form
$$
\sum_{k+l=n}   f^{k,l} \omega \Sn^k \Tn^l \qquad \text{ where } f^{k,l} \in \gr_{l-k+2} \Or(X)
$$
and elements of $\Gamma(Z, j^\ast\VV_n)^{\mmu_{12}}$ are of the form
$$
\sum_{k+l=n}   g^{k,l} j^\ast(\omega) \Sn^k \Tn^l \qquad \hbox{ where } g^{k,l} \in \gr_{[l-k+2]} \Or(Z),
$$
since $j^\ast(\omega)$ generates $\Omega^1(Z)$,  which follows from the fact that $\omega \wedge \psi\neq 0$ and therefore $j^\ast(\omega)\neq 0$. 
Injectivity follows from    Lemma~\ref{lem:Zisom}: if $j^\ast(f^{k,l})$ vanishes in $\gr_{[l-k+2]} \Or(Z)$, it follows that $f^{k,l}=0$. For the surjectivity, observe that 
$$
j^\ast(\sum_{k+l=n}   G^{k,l} \omega \Sn^k \Tn^l )
= \sum_{k+l=n}   g^{k,l} j^\ast(\omega) \Sn^k \Tn^l
$$
where $G^{k,l} \in \gr_{l-k+2} \Or(X)$ is the unique preimage of $g^{k,l} \in \gr_{[l-k+2]} \Or(Z)$. 
 \end{proof}

Since the map $Z \to \mmu_{12}\bbs Z = \M_{1,1}$ is an \'etale map of stacks over $\Q$, the complex
\begin{equation}
\label{eqn:DR_Z}
\Gamma(Z,\Omega^\bdot_Z\otimes j^\ast\VV_n)^{\mmu_{12}}
\end{equation}
computes $H^\bdot_\dR(\M_{1,1},\VV_n)$ for all $n\ge 0$. Thus we have:

\begin{cor}
The complex $\big(\Omega^\bdot(X,\VV_n)^\omega,\nabla^\omega\big)$ computes $H^\bdot_\dR(\M_{1,1},\VV_n)$ for all $n\ge 0$.
\end{cor}

We conclude this section by showing that the isomorphism (\ref{eqn:coho_splitting}) lifts to the level of de~Rham complexes.

\begin{prop}
\label{prop:quism}
There is a canonical isomorphism
$$
\Omega^\bdot(X,\VV_n)^\Gm \cong \Omega^\bdot(Z,j^\ast\VV_n)^{\mmu_{12}} \oplus \Omega^\bdot(Z,j^\ast\VV_n)^{\mmu_{12}}[-1]
$$
of complexes.
\end{prop}

\begin{proof}
The quotient of $\Omega^\bdot(X,\VV_n)^\Gm$ by $\Omega^\bdot(X,\VV_n)^\omega$ is $\psi\otimes \Omega^\bdot(X,\VV_n)^{\omega}$:
$$
\xymatrix@C=20pt@R=16pt{
0 \ar[r] & \Gamma(X,\VV_n)^\Gm \ar[r]^(.45){\nabla^\omega} \ar@{=}[d] & \Gamma(X,\omega\otimes\VV_n)^\Gm \ar[r] \ar[d] & 0 \ar[d]
\cr
0 \ar[r] & \Gamma(X,\VV_n)^\Gm \ar[r]^(.4){\nabla}\ar[d] & \Gamma(X,\Omega^1_X \otimes \VV_n)^\Gm \ar[r]^\nabla \ar[d] & \Gamma(X,\Omega^2_X \otimes \VV_n)^\Gm \ar[r] \ar@{=}[d] & 0
\cr
0 \ar[r] & 0 \ar[r] & \psi\otimes\Gamma(X,\VV_n)^\Gm \ar[r]^(.45){\nabla^\omega} & \psi\otimes\Gamma(X,\omega\otimes\VV_n)^\Gm \ar[r] & 0
}
$$
This exact sequence of complexes is naturally split since the third row defines a subcomplex of the second. The result follows  from lemma~\ref{lem:Zrestn}  since  $\psi\otimes \Omega^\bdot(X,\VV_n)^{\omega} \cong \Omega^\bdot(X,\VV_n)^{\omega}[-1]$.
\end{proof}

\section{Proof of Theorem~\ref{thm: main}}
 
In this section, we work over $\kk=\Q$. 
\subsection{Heads and tails}

We show that the ``head'' of an element $f$ of $\Gamma(X,\omega\otimes\VV_n)^\Gm$ is related to to the ``tail'' of $\nabla f$ by the Bol operator.

\begin{lem}
\label{lem: populate}
For all $n\geq 0$ there exists a unique $\Q$-linear map 
$$
\phi: M^!_{-n}  \To \Gamma (X, \VV_{n})^\Gm
$$
such that if we write 
$$
\phi = \sum_{j+k=n}  \phi^{j,k} \Sn^j \Tn^k
$$
where $\phi^{j,k} \in  \Hom(M^!_{-n},\gr_{k-j}\Or(X))$, then we have
\begin{equation}
\label{phitwoconditions} 
\phi^{n,0} (f) =  f  \quad \text{ and } \quad  \nabla\, \phi(f) \in \omega\Or(X) \Tn^n.
\end{equation} 
In other words, given a weakly holomorphic modular form $f$ of weight $-n$, there is a unique section of $\VV_{n}$ which coincides with $f$ in the first component, and whose image under $\nabla$ vanishes in all components save the last. 
\end{lem}

\begin{proof}
Suppose that $f\in M^!_{-n}$. We shall construct the components $f^{j,k} := \phi^{j,k}(f)$ inductively in $k$. For $k=0$, we have $f^{n,0}=f$.  The connection acts via (\ref{nablaomega}) which we write 
$$
\nabla^{\omega} = \omega \Big(\vartheta + \Sn \frac{\partial}{ \partial \Tn} 
 - \frac{u}{12} \, \Tn \frac{\partial}{\partial \Sn} \Big)\   .
$$
Suppose that $f^{a,b}$ is defined for $b \leq k < n$. 
The coefficient of  $\Sn^j \Tn^k$ in $\nabla^{\omega} \phi(f)$ is 
$$
\omega \Big(\vartheta(f^{j,k}) + (k+1)\, f^{j-1,k+1}  - (j+1)\frac{u}{12} f^{j+1,k-1} \Big) 
$$
There is a unique $f^{j-1,k+1} \in \gr_{k-j+2} \Or(X)$ that makes this vanish; namely
$$
f^{j-1,k+1} = 
\frac{1}{k+1} \Big( \frac{j+1}{12} \, u f^{j+1,k-1}- \vartheta(f^{j,k}) \Big)\ .
$$
By induction, these equations determine $\phi(f)$ uniquely. 
\end{proof}

Note that the inductive definition of $\phi$ involves dividing by $k+1$ for $1\leq k <n$. 

\begin{lem}[heads and tails]
\label{lem: phiisBol}
The diagram
$$
\xymatrix{
M_{-n}^!  \ar[r]^(.4)\phi \ar[d]_{{\D^{n+1}/n!}} &  \Gamma(X,\VV_{n})^\Gm \ar[d]^\nabla
\cr
M^!_{n+2}  \ar[r]^(.38)\bw  &  \Gamma(X,\omega\otimes \VV_{n})^\Gm
}
$$
commutes for all $n\ge 0$, where $\D^{n+1}$  is the Bol operator.
\end{lem}

\begin{proof}
Let $f \in M_{-n}^!$, and let $\phi(f)$ be the unique section constructed in Lemma~\ref{lem: populate}, whose  coefficient of $\Sn^n$ is $f$. Perform the change of gauge (\ref{STtoAT}). In this gauge, the coefficient of $\An^n$ in $\phi(f)$ is $f$:
$$
\phi(f) = \sum_{j+k=n} F^{j,k} \An^j \Tn^k, \text{ where }  F^{n,0} = f \ .
$$ 
The defining property of $\phi$ is that $\nabla \phi(f)$ is a multiple of $\omega \Tn^n$. Let $r(f)\in \Or(X)$ be the coefficient. That is,
$$
\nabla \phi(f) = r(f)\omega\, \Tn^n.
$$
This condition is preserved under the change of gauge (\ref{STtoAT}), so that 
$$
\nabla \phi(f) = r(f)\omega\, \Tn^n  = r(f)\, \frac{dq}{q} \Tn^n =  2 \pi i \, r(f)  dz\Tn^n
$$
By Lemma~\ref{lem: nablafAT}, we obtain the system of equations 
\begin{align*}
F^{n,0} & = f \\ 
\D F^{j,k} + (k+1)F^{j-1,k+1} & = 0 \\
 \D F^{{0,n}}  & = r(f) 
\end{align*}
It follows that $r(f) =  (-1)^n  \D^{n+1} f/n!$. This proves the result since  if $f$ is non-zero  $n$ must be  even  and $(-1)^n=1$. 
\end{proof}

\begin{rem}
The previous two lemmas imply a relation between the Bol operator, multiplication by  $u$, and the Serre derivative $\vartheta$. Compare \cite{SwinnertonDyer} (25).
\end{rem} 

We have therefore proved the existence of (\ref{omegamap}). It remains to prove that it is an isomorphism. 

\subsection{Proof of injectivity of (\ref{omegamap})}
Suppose that $\omega_g \in \Gamma (X, \omega\otimes\VV_n)^\Gm$ is exact. Write $\omega_g = \nabla f$, where $f\in \Gamma (X, \VV_n)^\Gm$. Consider the linear map 
\begin{align*}
\Gamma (X, \VV_n)^\Gm &\To  M_{-n}^!  \\
\sum_{j+k=n} f^{j,k} \Sn^j \Tn^k & \mapsto   f^{n,0}
\end{align*} 
Since $\omega_g$ has the property that all coefficients $\Sn^j \Tn^k$ except for $j=0,\ k=n$ vanish,
it follows from the uniqueness in Lemma~\ref{lem: populate} that $f= \phi(f^{n,0})$.  By Lemma~\ref{lem: phiisBol}, it follows that 
$\omega_g$ is in the image of $D^{n+1} M_{-n}^!$. 
 
\subsection{Proof of surjectivity (\ref{omegamap})}
\label{sec:surjectivity}

To complete the proof, we must show that the map (\ref{omegamap}) is surjective. By the algebraic de~Rham Theorem (Section~\ref{sec:DR}), every  class in $H^1_\dR(\M_{1,1};\VV_n)$ is represented by a section of the form 
$$
\eta = \sum_{j+k=n} f^{j,k}  \omega\,  \Sn^j \Tn^k,
$$
where $f^{j,k} \in \gr_{k-j+2} \Or(X)$.  We show by induction that such a form is equivalent, modulo the image of $\nabla$, to one in which $f^{j,k}$ vanishes for all $j>0$. Suppose that $0 \leq j \leq n$ is largest such that $f^{j,k}$ is non-zero.  If $j$ is zero then there is nothing to prove, so assume $j>0$. It follows from the connection formula (\ref{nabladef}) that
$$
\nabla  ( g \, \Sn^{j-1} \Tn^{k+1} ) \equiv (k+1)\, g\, \omega\,\Sn^{j} \Tn^k  \quad \bmod{\Tn^{k+1}}
$$
for any $g\in \Or(X)$ of weight $k-j+2$. Therefore, by replacing $\eta$ by 
$$
\eta - \frac{1}{(k+1)}\, \nabla  (f^{j,k} \,\Sn^{j-1} \Tn^{k+1} ) ,
$$
we can assume that $f^{j,k}$ vanishes, since the second term  lies in $\Gamma(X,\omega\otimes\VV_n)^\Gm$ by Lemma~\ref{lem: nologD}.
Proceeding in this manner, we deduce that every cohomology class in $H^1_\dR(\M_{1,1}; \VV_n)$ is represented by a form   in the image of (\ref{omegadef}). 

\begin{rem}
The above  argument  works in characteristic $0$ or $> \max\{3, n\}$. 
\end{rem}

\section{Periods of Cusp Forms}

Let $f \in S_n$ and $ g\in S_n^!$. Write $ f=\sum_{m>0} a_m q^m $ and $g= \sum_m b_m q^m$. In \cite[Thm.~1]{Guerzhoy}, a  Hecke-equivariant pairing is defined, up to a sign, by 
\begin{equation}
\label{eqn:guerzhoy}
\{ f, g\} = \sum_{m\in \Z} \frac{ a_m b_{-m}}{m^{n-1}} \ .
\end{equation}
In this section we show it extends to all $f\in S_n^!$ and that it corresponds to the image of $\omega_f\otimes \omega_g$ under the de~Rham incarnation of the cup product
$$
H^1_\cusp(\M_{1,1},\V^B_n) \otimes H^1_\cusp(\M_{1,1},\V^B_n) \To H^2_\cusp(\M_{1,1},\Q) \overset{\simeq}{\To} \Q
$$
induced by the natural pairing $\V^B_n\otimes \V^B_n \to \Q$, where $\V^B_n$ is the local system (\ref{eqn:Vn}).

\subsection{Inner products}

Let $V^\dR = \Q \An \oplus \Q\Tn$ be the fiber of the vector bundle $\VVbar$ over the cusp.  At $q=0$,
$$
\Sn = \An + 2\,\GE_2(0)\Tn = \An - \frac{1}{12}\Tn 
$$
so $V^\dR$ is also spanned by $\Sn$ and $\Tn$. Define a skew symmetric inner product
$$
\langle \blank,\blank \rangle_\dR : V^\dR\otimes V^\dR \to \Q
$$
by declaring that $\langle \Tn,\An \rangle = 1$. This is the natural inner product on $V^\dR$ and corresponds to the cup product pairing on the first de~Rham cohomology group of an elliptic curve. (Cf.\ \cite[Prop.~19.1]{KZB}.) It extends to a $(-1)^n$ symmetric inner product on $V^\dR_n := \Sym^n V^\dR$ by
$$
\langle v_1v_2\dots v_n,w_1 w_2\dots w_n \rangle
= \sum_{\sigma \in \Sigma_n} \prod_{j=1}^n \langle v_j,w_{\sigma(j)} \rangle
\qquad v_j,\ w_k \in V^\dR.
$$
In particular,
\begin{equation}
\label{eqn:inner_prod}
\langle \An^n,\Tn^n\rangle = \langle \Sn^n,\Tn^n\rangle = (-1)^n n!
\end{equation}

This inner product induces an inner product $\langle \blank,\blank\rangle_\dR : \VVbar_n \otimes \VVbar_n \to \Or_X$ which is flat with respect to the connection $\nabla$.

There is also a Betti version. As in the introduction, we set
$$
V^B = \Q \ba \oplus \Q\bb
$$
and $V_n^B = \Sym^n V_B$. The unique skew symmetric inner product $\langle\blank,\blank\rangle_B$ on $V_B$ satisfying $\langle \ba,\bb \rangle_B = 1$
induces a $(-1)^n$ symmetric inner product
$$
\langle \blank,\blank \rangle : V^B_n \otimes V^B_n \to \Q
$$
satisfying $\langle \ba^n,\bb^n\rangle = n!$ as above.

For each tangent vector $\vv = e^\lambda\partial/\partial q$ of the origin of the $q$-disk there is a comparison isomorphism
$$
\comp_{B,\dR} : V^\dR\otimes \C \To V^B\otimes \C,
$$
which is defined by
$$
\comp_{B,\dR}(\An) = \ba,\quad \comp_{B,\dR}(\Tn) = -2\pi i\,\bb + \lambda \ba.
$$
It corresponds to the limit mixed Hodge structure on the first cohomology of the first order smoothing of the nodal cubic in the direction of $\vec{v}$. Observe that, for all $\lambda$, two inner products are related by
\begin{equation}
\label{eqn:comp_of_inner_prods}
\comp^\ast_{B,\dR}\langle \blank,\blank\rangle_\dR
= (2\pi i)^{-n} \langle \blank,\blank \rangle_B.
\end{equation}
via the the comparison map $\comp_{B,\dR} : V_n^\dR\otimes \C \to V_n^B\otimes \C$.

\subsection{Residue maps}

The pullback of $(\VV_n,\nabla)$ to a formal neighbourhood of the origin of the $q$-disk $\DD^{\ast}$ along the map
$$
\rho :  \DD^\ast \To  X(\C)
$$
defined in \S \ref{sect: UHplane} is the free $\Q[[q]]$-module with basis $\{\An^n,\An^{n-1}\Tn,\dots,\Tn^n\}$ endowed with the connection (\ref{nablaAT}). It can also be expressed in the frame $\{\Sn^j\Tn^k\}$ using the formal change of gauge (\ref{STtoAT}).

The fraction field of $\Q[[q]]$ is the ring $\Q\lp q\rp := \Q[[q]][q^{-1}]$ of Laurent series. Set
$$
\Omega^1(\DD^\ast,\VV_n) := \Q\lp q\rp dq \otimes V^\dR_n.
$$
Define the local residue map $\Omega^1(\DD^\ast,\VV_n) \rightarrow \Q$
to be the composite 
\begin{equation} 
\label{eqn:residue}
\Res: \Omega^1(\DD^\ast,\VV_n) \To  V_n^\dR \To V_n^\dR/\An V_{n-1}^\dR \overset{\simeq}{\To} \Q
\end{equation}
of the usual residue map with the map that sends $\An$ to $0$ and $\Tn$ to $1$. That is,
$$
\Res  \,\big(  \sum_{j+k=n} \sum_{m\in\Z}a_m^{j,k} q^m \frac{dq}{q}\An^j\Tn^k \big) = 
a_0^{0,n}.
$$
Observe that if $f \in \Q\lp q\rp \otimes V_n$, then $\Res(\nabla f) = 0$,  since by equation (\ref{nablaAT}), an exact section satisfies $\nabla f \equiv 0 \bmod \An$.

The residue map
$$
\Res : \Gamma(X,\omega\otimes\VV_n)^\Gm \To \Q
$$
is defined to be the composite of the restriction map 
$$
\rho^{\ast}  : \Gamma(X,\omega\otimes\VV_n)^\Gm \To \Omega^1(\DD^\ast,\VV_n) \ .
$$
followed by the residue map (\ref{eqn:residue}).

\subsection{Cuspidal de~Rham cohomology}

Define the local cuspidal de~Rham complex $\Omega^\bdot(\DD^\ast,\VV_n)$ by
$$
\Omega^0_\cusp(\DD^\ast,\VV_n) :=
\Big\{
\sum_{j+k=n} a^{j,k}_m q^m \An^j\Tn^k : a^{n,0}_0 = 0
\Big\}
$$
$$
\Omega^1_\cusp(\DD^\ast,\VV_n)
:= \ker\Big\{\Res : \Q\lp q \rp\frac{dq}{q}\otimes V^{\dR}_n \to \Q\Big\}.
$$
It is closed under the differential $\nabla$. Since $H^0(\Omega^\bdot(\DD^\ast,\VV_n))$ is spanned by $\An^n$, we see that $H^0(\Omega^\bdot_\cusp(\DD^\ast,\VV_n))$ vanishes. The following lemma implies that this complex is acyclic. It is a local version of Lemma~\ref{lem: phiisBol}. Its proof is similar.

\begin{lem}[local heads and tails]
\label{lem:local}
For all $n\ge 0$, every element of $\Omega_\cusp^1(\DD^\ast,\VV_n)$ is exact in $\Omega^\bdot_\cusp(\DD^\ast,\VV_n)$. If
$$
f = \sum_{j+k=n} f^{j,k}\An^j\Tn^k \in \Q\lp q\rp V^{\dR}_n
 \quad \text{ and } \quad  \nabla(f) = \sum_{k\neq 0} a_k q^k \frac{dq}{q} \Tn^n,
$$
then
$$
(-1)^n n! \Big(q \frac{d}{dq}\Big)^{n+1} f^{n,0}
= \sum_{k\neq 0} a_k q^k.
$$
Consequently, the head $f^{n,0}$ of $f$ is 
$$
f^{n,0} = \frac{(-1)^n}{n!} \sum_{k\neq 0} \frac{a_k q^k}{k^{n+1}}.
\qed
$$
\end{lem}

The restriction mapping $\rho: \Omega^\bdot(X,\VV_n)^\omega \to \Omega^\bdot(\DD^\ast,\VV_n)$ commutes with $\nabla$. Consequently,
$$
\Omega^\bdot_\cusp(X,\VV_n)^\omega := \rho^{-1}\big(\Omega^\bdot_\cusp(\DD^\ast,\VV_n)\big)
$$
is a subcomplex of $\Omega^\bdot(X,\VV_n)^\omega$.

\begin{lem}
The complex $\Omega^\bdot_\cusp(X,\VV_n)^{\omega}$ computes the cuspidal de~Rham cohomology of $\M_{1,1}/\Q$. That is, the comparison isomorphism
$$
\comp_{B,\dR} : H^\bdot_\dR(\M_{1,1},\VV_n)\otimes\C \overset{\sim}{\To} H^\bdot(\SL_2(\Z),V^B_n)\otimes\C
$$
restricts to an isomorphism
$$
\comp_{B,\dR} : H^\bdot_\cdR(\M_{1,1},\VV_n)\otimes\C \overset{\sim}{\To} H^\bdot_\cusp(\SL_2(\Z),V^B_n)\otimes\C.
$$
\end{lem}

\begin{proof}
This follows directly from the exactness of the sequence 
$$
0 \to \Omega^\bdot_\cusp(X,\VV_n)^{\omega} \to \Omega^\bdot(X,\VV_n)^{\omega} \to  \big[\Q\An^n \overset{0}{\to} \Q \frac{dq}{q}\Tn^n \big]  \to 0.
$$
Alternatively, it can be deduced directly from Theorem~\ref{thm: main}.
\end{proof}

\subsection{The residue pairing}

By Lemma~\ref{lem:local}, the first homology of $\Omega^{\bdot}_\cusp(\DD^\ast,\VV_n)$ vanishes. Define the local residue pairing
$$
\{\blank,\blank\} : \Omega_\cusp^1(\DD^\ast,\VV_n) \otimes \Omega^1(\DD^\ast,\VV_n) \to \Q
$$
as follows. Lemma~\ref{lem:local} implies that each $\xi \in \Omega_\cusp^1(\DD^\ast,\VV_n)$ can be uniquely written $\xi = \nabla F$, where $F\in \Omega^0_\cusp(\DD^\ast,\VV_n)$. We set
$$
\big\{\xi,\eta\big\}  := \big\{ \nabla F, \eta \big\} =  \Res_{q=0}  \langle  F, \eta \rangle \ .
$$

\begin{lem}
\label{lem:formula}
This pairing is well-defined. We have
$$
\Big\{\sum_{j\neq 0} a_j q^j \frac{dq}{q} \Tn^n,\sum_{k\in\Z} b_k q^k \frac{dq}{q} \Tn^n\Big\}
= \sum_{k+l=0} \frac{a_k b_l}{k^{n+1}}. 
$$
It is $(-1)^{n+1}$ symmetric when restricted to $\Omega^1_\cusp(\DD^\ast,\VV_n)^{\otimes 2}$.
\end{lem}

\begin{proof}
That it is well-defined follows from the uniqueness of $F$. To see that it is $(-1)^{n+1}$ symmetric on $\Omega^1_\cusp(\DD^\ast,\VV_n)$, write $\xi = \nabla F$ and $\eta = \nabla G$, where $F,G \in \Omega^0_\cusp(\DD^\ast,\VV_n)$. By flatness of the inner product, 
$$
d \langle F, G\rangle =
\langle \nabla F, G \rangle - (-1)^n \langle \nabla G, F \rangle.
$$
The symmetry property follows as the left-hand side has vanishing residue.

The formula is an immediate consequence of Lemma~\ref{lem:local} as
\begin{align*}
\Big\{\nabla\Big(\sum_{j+k=n} f^{j,k}\An^j\Tn^k\Big),  g\, dq\,  \Tn^n\Big\}
&= \sum_{j+k=n} \Res_{q=0} f^{j,k}g \langle \An^j\Tn^k, \Tn^n \rangle
\cr
&=    (-1)^n  n! \Res_{q=0} f^{n,0}g.
\end{align*}
\end{proof}

By composing with the restriction maps
$$
\Omega^1_\cusp(X,\VV_n)^\omega \to \Omega^1_\cusp(\DD^\ast,\VV_n)
\text{ and }
\Omega^1(\M_{1,1},\VV_n) \to \Omega^1(\DD^\ast,\VV_n)
$$
we obtain a well defined pairing
\begin{equation}
\label{eqn:pairing}
\{\blank,\blank\} : \Omega^1_\cusp(X,\VV_n)^{\omega} \otimes \Omega^1(\M_{1,1},\VV_n) \to \Q.
\end{equation}

\begin{lem}
The pairing (\ref{eqn:pairing}) has the property that $\{\nabla f,\xi\} = 0$ for all $f\in \Omega^0(X,\VV_n)^\omega$ and $\xi \in \Omega^1(X,\VV_n)^\omega$. Consequently, it induces a well-defined pairing
$$
\textstyle{\int^\dR} : H^1_\cdR(\M_{1,1},\VV_n) \otimes H^1_\dR(\M_{1,1},\VV_n) \to \Q
$$
such that the diagram
$$
\xymatrix{
S^!_n \otimes M_n^! \ar[r]^{\{\blank,\blank\}}\ar[d]_{\bw\otimes\bw} & \Q  \cr
H^1_\cdR(\M_{1,1},\VV_n) \otimes H^1_\dR(\M_{1,1},\VV_n) \ar[ur]_(.55){\int^\dR}
}
$$
commutes.
\end{lem}

\begin{proof}
The one-form $\langle f,\xi\rangle$ is an element of $\Gamma(X, \omega\otimes\VV_0)^\Gm$, and therefore defines a class in the cohomology of the complex
(\ref{eqn:complex}). We showed that the latter computes $H^1_\dR(\M_{1,1}; \VV_0) = 0$, and so $\langle f,\xi\rangle$  is  exact. Its residue therefore vanishes.

Since the pairing (\ref{eqn:pairing}) is $(-1)^{n+1}$ symmetric on $\Omega^1_\cusp(\M_{1,1},\VV_n)$, and since $\nabla f \in \Omega^1_\cusp(\M_{1,1},\VV_n)$ for all $f\in \Omega^0(X,\VV_n)$, this implies that $\{\xi,\nabla f\}=0$ for all $\xi \in \Omega^1_\cusp(\M_{1,1},\VV_n)$ and that the pairing is well-defined on cohomology.

The formula for the pairing is a direct consequence of Lemma~\ref{lem:formula}.
\end{proof}

\subsection{Relation to cup product}

Our final task is to determine how $\textstyle{\int^\dR}$ is related to the cup product. For this we need to discuss relative cohomology and its relation to cuspidal cohomology. Let $\Gamma_\infty$ be the subgroup of $\SL_2(\Z)$ generated by {\tiny${\begin{pmatrix} 1 & 1\cr 0 & 1\end{pmatrix}}$.} Throughout this section, we assume that $n>0$.

We have the exact sequence
\begin{equation}
\label{eqn:ses_cuspidal}
0 \To H^0(\Gamma_\infty;V^B_n) \To H^1(\SL_2(\Z),\Gamma_\infty;V^B_n) \To H^1_\cusp(\SL_2(\Z);V^B_n) \To 0
\end{equation}
which is dual to the exact sequence
$$
0 \To H^1_\cusp(\SL_2(\Z);V_n^B) \To H^1(\SL_2(\Z);V_n^B) \To H^1(\Gamma_\infty;V_n^B) \To 0
$$
under the cup product pairing
$$
\textstyle{\int^B} : H^1(\SL_2(\Z),\Gamma_\infty;V_n^B) \otimes H^1(\SL_2(\Z);V_n^B)
\To H^2(\SL_2(\Z),\Gamma_\infty;\Q) \cong \Q
$$
induced by $\langle \blank,\blank\rangle_B$.

\begin{rem}
It is helpful to note that $H^0(\Gamma_\infty,V_n^B) = \Q\ba^n$ and $H^1(\Gamma_\infty,V_n^B) = V_n^B/\ba V_{n-1}^B \cong \Q\bb^n$.
\end{rem}

The algebraic analogue of the relative cohomology group above is the de~Rham cohomology $H^\bdot_\dR(\M_{1,1},\DD^\ast;\VV_n)$, which is defined to be the cohomology of the complex
$$
\Omega^\bdot(X,\DD^\ast;\VV_n)^\omega :=
\cone\big[\Omega^\bdot(X,\VV_n)^\omega \To \Omega^\bdot(\DD^\ast,\VV_n)\big][-1]
$$
There is a comparison isomorphism
$$
\comp_{B,\dR} : H^1_\dR(\M_{1,1},\DD^\ast;\VV_n)\otimes \C
\To H^1(\SL_2(\Z),\Gamma_\infty;V_n^B)\otimes \C.
$$
There is a short exact sequence
\begin{equation}
\label{ses:deRham}
0 \To \Q \An^n \To H^1_\dR(\M_{1,1},\DD^\ast;\VV_n) \To H^1_\cdR(\M_{1,1},\VV_n) \To 0
\end{equation}
which maps to (\ref{eqn:ses_cuspidal}) after tensoring both sequences with $\C$. 

In order to use the cup product to construct a pairing between $H^1_\cdR(\M_{1,1},\VV_n)$ and $H^1_\dR(\M_{1,1},\VV_n)$, we need to choose a splitting of (\ref{ses:deRham}). Here we use the splitting
$$
s : H^1_\cdR(\M_{1,1},\VV_n) \To H^1_\dR(\M_{1,1},\DD^\ast;\VV_n)
$$
that is defined by taking the class of $\xi \in \Omega^1_\cusp(X,\VV_n)^\omega$ to the class of $(\xi,F)$, where $F$ is the unique element of $\Omega^0_\cusp(\DD^\ast,\VV_n)$ whose derivative is the restriction of $\xi$ to $\DD^\ast$.

\begin{prop}
\label{prop:comparison}
The diagram
$$
\xymatrix{
H^1_\cdR(\M_{1,1},\VV_n) \otimes H^1_\dR(\M_{1,1},\VV_n)\otimes\C
\ar[d]_{\comp_{B,\dR}^{\otimes 2}\circ(s\otimes 1)} \ar[r]^(.82){\int^\dR} & \C\,\e^\dR \ar[d]
\cr
H^1(\SL_2(\Z),\Gamma_\infty;V_n)\otimes H^1(\M_{1,1}^\an;\V_n)\otimes\C \ar[r]^(.82){\int^B} & \C\,\e^B
}
$$
commutes, where $\e^\dR = (2\pi i)^{n+1}\e^B$ and where $\textstyle{\int^B}$ denotes the cup product induced by $\langle \blank,\blank\rangle_B$ evaluated on the fundamental class of $\Mbar_{1,1}(\C)$.
\end{prop}

\begin{proof}
Suppose that $\xi \in \Omega_\cusp^1(X;\VV_n)$ is cuspidal and $\eta \in \Omega^1(X, \VV_n)^\omega$. They represent cohomology classes. Let $U$ be the analytic $q$-disk $\{q \in \C : |q| < e^{-2\pi} \}$ and $U' = U-\{0\}$. Since $\xi$ is cuspidal, its restriction to $U'$ is exact. Choose a meromorphic section $\Ftilde\in \Gamma(U',\VV_n^\an)$ such that $\nabla \Ftilde = \xi$ on $U'$ and the image $F$ of $\Ftilde$ in $\Omega^0(\DD^\ast,\VV_n)$ is the unique element of $F\in \Omega^0_\cusp(\DD^\ast,\VV_n)$ such that $\nabla F$ is the $q$-expansion of $\xi$.

Choose $r,R\in \R$ such that $0 < r < R < e^{-2 \pi}$. Choose a smooth function $\rho: U \to \R_{\ge 0}$ which vanishes outside the annulus
$$
A := \{q \in U : r \le |q| \le R\}
$$
and is equal to $1$ identically when $|q|\le r$. Then $\xitilde := \xi - \nabla(\rho \Ftilde)$ extends by $0$ to a smooth 1-form on the orbifold $\Mbar^\an_{1,1}$ 
with values in $\VV_n\otimes\C$, which vanishes in a neighbourhood of the cusp and equals $\xi$ on $|q| > R$.

Since $\rho \Ftilde$ is a smooth section of $\VV_n\otimes \C$ over $\M_{1,1}^\an$, $\xitilde$ is a smooth form that represents the same class in $H^1_\cusp(\M_{1,1},\VV_n)$ as $\xi$. Since $\nabla(\rho \Ftilde)$ is supported in the annulus $A$ and since $\xi\wedge\eta=0$, $\xitilde\wedge \eta$ is supported in $A$. Using (\ref{eqn:comp_of_inner_prods}), we have
\begin{multline*}
\textstyle{\int^B}(\xi \otimes \eta)
= \int_{\Mbar_{1,1}^\an} \langle \xitilde, \eta\rangle_B 
= (2\pi i)^n\int_{\Mbar_{1,1}^\an} \langle \xitilde, \eta\rangle_\dR
= -(2\pi i)^n\int_A d\langle \rho \Ftilde,\eta\rangle_\dR
\cr
\overset{\text{(Stokes)}}{=} (2\pi i)^n \int_{|q|=r} \langle \Ftilde,\eta\rangle_\dR
= (2\pi i)^{n+1}\Res_{q=0} \Ftilde\eta
\cr
= (2\pi i)^{n+1} \{\xi,\eta\}
= (2\pi i)^{n+1}\textstyle{\int^\dR}(\xi\otimes \eta).
\end{multline*}
\end{proof}

\begin{rem}
The Hecke invariance of Guerzhoy's inner product (\ref{eqn:guerzhoy}) on cusp forms follows directly from this description of the inner product using the projection formula.\footnote{This states that if $f : X \to Y$ is a smooth proper morphism, then $(f_\ast a)\cdot b = f_\ast(a \cdot f^\ast b)$, where $a\in H^\bdot(X;f^\ast\V)$ and $b\in H^\bdot(Y,\V)$.}
\end{rem}

\begin{rem}

The section $s$ defined above cannot be Hecke invariant. If it were, the Eisenstein series $\GE_{n+2}$ would be orthogonal to $S_{n+2}^!$ under the inner product $\{\blank,\blank\}$. However, this is not the case, since it would contradict the discussion in \S\ref{sect: split}: orthogonality with respect to $\GE_{n+2}$ and cuspidality are two distinct linear conditions  on the space of weakly-holomorphic modular forms.

Consider, by way of example, the case $n=10$. There is a unique  $\Q$-linear combination $f_{-1}\in S_{12}^!$  of the weakly holomorphic modular forms $\GE_{24}\Delta^{-1}$,  $\GE_{12}$ and $\Delta$ such that $f_{-1} = q^{-1} + O(q^2)$. It is given explicitly by
$$
f_{-1} = q^{-1} +47709536 \, q^2+39862705122\, q^3+7552626810624\, q^4+ \cdots \ .
$$
Its class in $S_{12}^! / \D^{11} M_{-10}^!$ is a Hecke eigenform, with the same eigenvalues as $\Delta$. Since the leading coefficient of $\Delta$ is $q$, we have 
$$
\{ \Delta, f_{-1}  \} = 1
$$
On the other hand, from the Fourier expansion
$$
\GE_{12} =   \frac{691}{65520}+q+2049 q^2+177148 q^3+\cdots \ ,
$$ 
we find that  a naive application of the formula (\ref{eqn:guerzhoy}) to $f_{-1}$ and $\GE_{12}$ would give $\{\GE_{12}, f_{-1} \}=1$, and hence $S_{12}^!$ is not orthogonal to $\GE_{12} \Q$.  
\end{rem}

\subsection{Proof of Theorem~\ref{thm: determinant}}

Suppose that $f$ is a Hecke eigen cusp form of weight $2n$. Denote the associated 2-dimensional subspace of $H^1(\M_{1,1},\VV_{2n-2})\otimes K_f$ by $V_f$. It has Betti and de~Rham realizations $V_f^B$ and $V_f^\dR$ related by the comparison isomorphism $\comp_{B,\dR} : V_f^\dR\otimes_{K_f}\C \to V_f^B\otimes_{K_f}\C$. The cup product induces a non-degenerate, skew-symmetric pairing
$$
\langle\blank,\blank\rangle : 
V_f \otimes V_f \to K_f(-2n+1).
$$
It has Betti and de~Rham realizations
$$
\langle\blank,\blank\rangle_B = \textstyle{\int^B} \otimes\, \e^B
\text{ and }
\langle\blank,\blank\rangle_\dR = \textstyle{\int^\dR} \otimes\, \e^\dR.
$$
Let $\alpha_f,\beta_f$ be a $K_f$-de~Rham basis of $V_f^\dR$. There is a $K_f$ basis $r_f^+,r_f^-$ of $V_f^B$ with $\langle r_f^+,r_f^-\rangle_B =\e_B$. Then
$$
\begin{pmatrix}
\alpha_f & \beta_f
\end{pmatrix}
=
\begin{pmatrix}
r_f^+ & r_f^-
\end{pmatrix}
\begin{pmatrix}
\eta_f^+ & \omega_f^+ \cr i \eta_f^- & i \omega_f^-
\end{pmatrix}
=
\begin{pmatrix}
r_f^+ & r_f^-
\end{pmatrix}
P_f
$$
where $\eta_f^\pm$ and $\omega_f^\pm$ are real numbers. By Lemma~\ref{prop:comparison},
\begin{align*}
\langle\comp_{B,\dR}(\omega_f),\comp_{B,\dR}(\eta_f) \rangle_B
&=
\langle \eta_f^+ r_f^+ + i\eta_f^-r_f^-, \omega_f^+ r_f^+ + i \omega_f^- r_f^- \rangle_B
\cr
&= \det (P_f) \, \e_B
\cr
&= (2\pi i)^{-2n+1} \det( P_f)  \, \e_\dR
\end{align*}
Since $\langle \omega_f, \eta_f \rangle_\dR \in K^{\times}_f\,\e_\dR$, this implies that $\det( P_f) \in (2\pi i)^{2n-1}K^{\times}_f$.

\section{A $\Q$-de~Rham splitting of the Hodge filtration}
\label{sect: split}

Let $\ell = \dim S_{n}$, so that
$$
\dim M_{n}^!/\D^{n-1} M_{2-n}^! = \dim H^1_\dR(\M_{1,1},\VV_{n-2}) = 2\ell + 1.
$$
Let $\mathrm{ord}_{\infty}$ denote the order of vanishing at the cusp. It follows from the Riemann--Roch formula, as noted in \cite{DukeJenkins}, that
$$
\mathrm{ord}_{\infty} f \leq \ell
$$
for all $f\in M^!_n$. Furthermore, it was shown in \cite{Guerzhoy} that for any $f\in M_{n}^!$, there exists a unique representative of $f$ modulo $\D^{n-1} M_{2-n}^!$ such that 
$$
\mathrm{ord}_{\infty} f\geq -\ell \ .
$$
Since the dimension of $M_{n}^!/\D^{n-1} M_{2-n}^!$ is exactly $2 \ell+1$, it follows that such an $f$ is uniquely determined by its $2 \ell +1$ Fourier coefficients $(a_{-\ell},\ldots, a_{\ell})\in \Q^{2\ell+1}$ where
$$
f= \sum_{n\geq -\ell}  a_n q^n \,
$$
and, conversely,  any vector $(a_{-\ell},\ldots, a_{\ell}) \in \Q^{2\ell+1}$ uniquely determines an element in $M_{n}^!/ D^{n-1} M_{2-n}^!$. It follows that the functions  $f\in M_{n}^!$ of the form 
$$
f_m=   q^{m} + O(q^{\ell+1})
$$
for every $-\ell  \leq m \leq \ell$, are  a  $\Q$-basis for  $H^1_\dR(\M_{1,1}; \VV_n)$, by Theorem~\ref{thm: main}. These functions satisfy some remarkable properties, studied in \cite{DukeJenkins}.

This basis simultaneously gives a $\Q$-de~Rham splitting of the Hodge filtration
$$
H^1_\dR(\M_{1,1},\VV_{n-2}) = F^0 \supset F^1 = \dots = F^{n-2} \supset F^{n-1} = M_n \supset F^n = 0
$$
and of the weight filtration
\begin{multline*}
0 = W_{n-2} \subset H^1_\cdR(\M_{1,1},\VV_{n-2}) = W_{n-1} \subset W_n = \cdots
\cr
\cdots = W_{2n-3} \subset W_{2n-2} = H^1_\dR(\M_{1,1},\VV_{n-2}).
\end{multline*}
The splitting of
\begin{multline*}
0 \to F^{n-1}H^1_\cdR(\M_{1,1},\VV_{n-2}) \to H^1_\cdR(\M_{1,1},\VV_{n-2})
\cr
\to \gr_F^{n-1} H^1_\cdR(\M_{1,1},\VV_{n-2}) \to 0
\end{multline*}
is given by lifting $\gr_F^0H^1_\cdR(\M_{1,1},\VV_{n-2})$ to the subspace of $M^!_n/\D^{n-1}M^1_{2-n}$ consisting of those $f$ whose Fourier coefficients $a_j$ vanish when $0\le j \le \ell$. The splitting of the weight filtration is given by the Eisenstein series.

\appendix

\section{Framings}
\label{sec:frames}

The aim of this appendix is to explain the choice of the power of $2\pi i$ in the cocycle formula (\ref{eqn:cocycle}). In this appendix we take $\kk = \Q$.

Recall that $\pi:\E \to X$ denotes the universal elliptic curve and that $\V^B_n$ denotes the $n$th symmetric power of $R^1 \pi_\ast \Q$. It underlies a variation of Hodge structure of weight $n$. To give a framing of $\V^B_n$, it suffices to give a framing of $\V^B := \V^B_1$. 

The pullback of $\V^B$ to $\HH$ along $\rho : \HH \to X$ is the trivial local system whose fiber over $z \in \HH$ is  $H^1(E_z)$, where $E_z := \C/(\Z \oplus z\Z)$. We identify $H^1(E_z)$ with its dual $H_1(E_z) \cong \Hom(H_1(E_z),\Z)$ by Poincar\'e duality
$$
\PD : H_1(E_z) \to H^1(E_z), \quad \PD(c) := \langle c,\blank\rangle,
$$
where $\langle \blank,\blank \rangle$ denotes the intersection pairing. On the level of Hodge structures, Poincar\'e duality is an isomorphism
$$
\PD : H_1(E_z,\Q) \to H^1(E_z,\Q(1)).
$$

Denote the standard basis of $H_1(E_z;\Z)$ by $\ba$, $\bb$. These classes correspond to the lattice points $1$ and $z$, respectively. Denote the dual basis of $H^1(E_z) \cong \Hom(H_1(E_z),\Z)$ by $\adual,\bdual$. Then $\bdual = \PD(\ba)$ and $\adual = \PD(\bb)$. We identify the (Betti) components of $H_1(E)$ and $H^1(E)$ via $\PD$. With this identification
$$
dw = -\bb + z\ba
$$
where $w$ is the coordinate in the universal covering $\C$ of $E_z$. The abelian differential $dx/y$ on the elliptic curve corresponding to $\rho(z)$ is $2\pi i dw$. This is the section $\Tn$ of $\VV_n$. So $\Tn$ corresponds to the section
$$
\Tn = 2 \pi i dw = 2\pi i(z\ba - \bb)
$$
of $\VV^\an := \V\otimes \Or(\HH)$.

Each $f\in M_{n+2}^!$ corresponds to an element $h(u,v) \in \Or(X)$. The corresponding 1-form $\omega_f$ is $h \Tn^n \omega$. So
$$
\rho^\ast \omega_f = (2\pi i)^n f(z) (z\ba - \bb)^n \frac{dq}{q} 
= (2\pi i)^{n+1} f(z) (z\ba - \bb)^n dz.
$$

\end{document}